\documentclass[11pt,leqno]{amsart}

\usepackage[all,cmtip]{xy}
\usepackage{a4,latexsym}
\usepackage[utf8]{inputenc}
\usepackage{amsfonts}
\usepackage[hidelinks]{hyperref}
\usepackage{amsxtra,amsmath,amsfonts,amscd, amssymb, mathrsfs, amsthm}
\usepackage{color, mathtools}
\usepackage{bbm}
\usepackage{enumitem}
\usepackage{xcolor}
\usepackage{graphicx}
\usepackage{tikz-cd}
\usepackage{braket}
\usepackage[all]{xypic}
\usepackage[toc,page,title,titletoc,header]{appendix}
\usepackage[capitalize]{cleveref}

\numberwithin{paragraph}{section}

\setlist[enumerate]{label=\it{(\roman*)},
	ref=\it{(\roman*)}}

\newcommand{\R}{{\mathbb R}}

\newcommand{\Q}{{\mathbb Q}}
\newcommand{\Z}{{\mathbb Z}}
\newcommand{\N}{{\mathbb N}}

\newcommand{\NN}{\mathcal{N}}

\newtheorem{theorem}{Theorem}[section]

\newtheorem{lemma}[theorem]{Lemma}
\newtheorem{lemma*}{Lemma}

\theoremstyle{definition}

\newtheorem{proposition&definition}[theorem]{Proposition\&Definition}
\newtheorem{lemma&definition}[theorem]{Lemma\&Definition}
\newtheorem{theorem&definition}[theorem]{Theorem\&Definition}

\newtheorem{example*}{Example}
\newtheorem{remark}{Remark}

\newtheorem{question*}{Question}

\def\I{\mathcal I}

\def\I{{\mathcal I}}
\def\T{{\mathcal T}}

 \def\NN{ N}
\def\ds{\displaystyle}

\numberwithin{equation}{section}

\begin{document}
	
	\title[ On Some Multipliers Related to Discrete Fractional Integrals]{ On Some Multipliers Related to Discrete Fractional Integrals }
	
	\author[J.~Cheng]{Jinhua Cheng}
	\address{J. Cheng,  School of  Mathematical Sciences, Zhejiang University, 310058 Hangzhou, P. R. China }
	\email{chengjinhua@zju.edu.cn}
	

	\begin{abstract}
This paper explores the properties of multipliers associated with discrete analogues of fractional integrals, revealing intriguing connections with Dirichlet characters, Euler's identity, and Dedekind zeta functions of quadratic imaginary fields. Employing Fourier transform techniques, the Hardy--Littlewood circle method, and a discrete analogue of the Stein--Weiss inequality on product space through implication methods, we establish $\ell^p\rightarrow\ell^q$ bounds for these operators. Our results contribute to a deeper understanding of the intricate relationship between number theory and harmonic analysis in discrete domains, offering insights into the convergence behavior of these operators.
	\end{abstract}
	
	
	\maketitle
	
	\setcounter{tocdepth}{1}
	

\section{Introduction}
\setcounter{equation}{0}
The study of discrete analogues in harmonic analysis indeed shares a companionable relationship with the early history of singular integrals. Singular integrals, which arise from the convolution of functions with singular or highly oscillatory kernels, have been a central focus of harmonic analysis since its inception.  For example, in 1928, M. Riesz
 \cite{R28}  proved  the Hilbert transform, 
\[
Hf(x) = \int_{\R}\frac{f(x-y)}{y} dy, 
\]
is bounded on $L^p(\R)$ for all $1<p< \infty$,  and this implies its discrete analogue,
\[
\mathcal{H} f(n) = \sum_{m\in \Z \atop m \neq 0} \frac{f(n-m)}{m},
\]
is bounded on $\ell^p(\Z)$  for all $1<p< \infty$.  Here, $\ell^p(\Z)$ is defined as
\[
 \ell^p(\Z) =\left \{f ~ \text{defined on}~  \Z |  \sum_{m \in \Z} |f(m)|^p  < \infty \right \}.
\]
Moreover,
  $||f ||_{\ell^p(\Z) }  = \left( \sum_{m \in \Z } | f(m) |^p \right)^{\frac{1}{p}} $.

Another classical family of operators in harmonic analysis are the fractional integral operators, 
\[
     I_{s}f(x)=\int_{\R} \frac{f(x-y)}{|y|^{s}}dy, \quad 0 < s  < 1.
\]
It  is well known that  for  $1< p<q < \infty $ with $1/ q = 1/p - (1 - s )  $, $I_s$ is a bounded operator from $L^p(\R)$  to
$L^q(\R)$. The discrete analogue of this operator is defined by 
\[
\mathcal{I}_s f(n) = \sum_{m\in \Z \atop m \neq 0} \frac{f(n-m)}{|m|^s}.
\] 
Similarly,  the boundedness of $I_s$ implies the boundedness of  $\I_s$.
Consider a function $f$ defined on $\mathbb{Z}$, where $0<s<1$ and $k\geq 1$ is an integer. The discrete fractional operator $\I_{s,k}$ is defined as follows:
\begin{equation}
    \I_{s,k} f(m)=\sum_{n=1}^{\infty} \frac{f(m-n^k)}{n^s} ,\quad m\in\Z
\end{equation}
acting on functions defined on $\Z$. Stein and Wainger \cite{Stein-Wainger1} initiated the study of the $~\ell^p\rightarrow \ell^q~$ boundedness of $\I_{s,k}$, that is, there exists some constant $C$ such that
\[
|| \I_{s, k} f ||_{\ell^q}  \le C ||f||_{\ell^p}.
\]
On  the other hand, for $f$ defined on $\Z$,  its Fourier transform is defined by $ \hat{f}(x) = \sum_{n \in \Z} f(n) e^{- 2\pi i nx } $. Therefore,

\[
\widehat{\I_{s,k} f}(x) =m_{s,k}(x)\Hat{f}(x), \qquad m_{s,k}(x)=\sum_{n=1}^\infty \frac{e^{-2\pi i n^k x}}{n^s}.
\]
Here,  $m_{s, k}(x)$ is called the Fourier multiplier.

They demonstrated that when $1/2<s<1$, $m_{s,2}$ belongs to weak-type $\mathrm{L}^{2/(1-s)}[0,1]$ and $m_{s,k}$ belongs to weak-type $\mathrm{L}^{k/(1-s)}[0,1]$ as long as $s$ is sufficiently close to $1$.  The main tool is the Hardy--Littlewood circle method; for a more detailed introduction on circle methods, see \cite{W21}.
Furthermore, if this holds for all $1/2<s<1$, it would imply the ``Hypothesis $K^*$'' of Hardy, Littlewood, and Hooley, which remains an open problem in number theory.

Lillian Pierce's thesis \cite{Pierce Lillian 3} extended this result to positive definite quadratic forms. For instance, let $Q(x)=\frac{1}{2}x^tAx$ be a positive definite quadratic form, where $A$ is a real, positive definite, $2\times 2$ symmetric matrix with integer entries and even diagonal entries. Then, the corresponding multiplier
\[
m_{s,Q}(x)=\sum_{m\in \Z^2 \atop m\neq 0} \frac{e^{-2\pi i Q(m)x}}{Q(m)^s}
\]
is of weak-type $\mathrm{L}^{1/(1-s)}[0,1]$. In this paper, we mainly consider the 'twisted' multiplier. Let ${a_n}$ be a complex series, and define the corresponding multiplier by
\begin{equation}\label{twisted multiplier1}
    m_{s,\{a_n\}}(x)=\sum_{n=1}^\infty \frac{a_ne^{-2\pi i nx}}{n^s}.
\end{equation}
It is worth noting that if we set $a_n=1$ for $n=m^k\geq 1$ and $a_n=0$ otherwise, then $m_{s,\{a_n\}}(x)=m_{sk,k}(x)$. The series ${a_n}$ can originate from various areas related to number theory.

For instance, in Section \ref{sec2}, we delve into a \textbf{primitive} Dirichlet character $\chi$ modulo $N$, defining the corresponding twisted multiplier as

\begin{equation}\label{twisted multiplier}
m_{s,\chi}(x)=\sum_{n=1}^{\infty} \frac{\chi(n)e^{-2\pi \mathrm{i} n^2x}}{n^s} ,
\end{equation}
which differs from $m_{s,2}(x)$ in several respects. For example, $m_{s,\chi}(0)$ corresponds to the Dirichlet $L$-function and is bounded when $0<s<1$, whereas $m_{s,2}(x)$ tends to infinity as $x$ approaches $0$. However, we will demonstrate the following Theorem.

\begin{theorem}
    For  $1/2<s<1$, let $\chi$ be a primitive Dirichlet character modulo $N$; then, $~m_{s,\chi}~$ belongs to weak-type $L^{2/(1-s)}[0,1] $.
\end{theorem}

In Section \ref{sec3}, we investigate the scenario where ${a_n}$ originates from Euler's identity, given by

\[
\prod_{n=1}^\infty (1-x^n)=\sum_{n=0}^\infty a_n x^n.
\]
We establish that $m_{s, \{a_n\}}$ belongs to weak-type $\mathrm{L}^{2/(1-2s)}[0,1]$ and provide an improved result regarding the regularity of the corresponding discrete fractional integral operator.

In Section \ref{sec4}, we delve into imaginary quadratic fields and the associated Dedekind zeta function. We demonstrate the close connection between the corresponding multipliers and positive definite quadratic forms as investigated by Lillian Pierce.

In the final section, we tackle the discrete analogue of the Stein--Weiss inequality on product space. Employing the "implication" method, we deduce the regularity property of the discrete fractional operator. 

Discrete analogues in harmonic analysis have garnered significant attention in recent decades, with notable contributions from scholars such as Stein and Wainger~\cite{Stein-Wainger1,Stein-Wainger2}, \mbox{Oberlin \cite{Oberlin}} and  Lillian Pierce \cite{Pierce Lillian 1,Pierce Lillian 2,Pierce Lillian 3} (see also \cite{B23, B19, M19, M24}). This paper introduces a novel perspective to the study of multipliers in harmonic analysis by incorporating primitive Dirichlet characters. This addition not only enriches the theoretical framework but also presents new challenges and complexities to be explored. As examples of the application of this approach, this paper investigates multipliers associated with Euler's identity and quadratic imaginary fields.

\begin{remark}
Let $r>0$ be a real number; we define a function $f$ as belonging to weak-type $ L^r[0,1]$ if
\[
|\{x\in [0,1]: f(x)>\alpha\}|\le c \alpha^{-r},\qquad\textit{for} \qquad  \alpha>0
\]
where $c$ is a constant independent of $\alpha$ and $f$.
\end{remark}

\section{Multipliers Twisted with Dirichlet Characters}\label{sec2}
Fix an integer $N>1$.
A {\bf Dirichlet character modulo
} $~N~$ is a function $\chi: ~\Z\rightarrow \mathbb{S}^1\cup \{0\}$ such that the following is true:
\begin{itemize}
    \item [(i)] $\chi(n)=0~$ if and only if $~(n,N)>1$. 
    \item [(ii)]  $\chi(n)=\chi(m)~$ if $~n\equiv m~$ mod$~N$.
     \item [(iii)]  $\chi(mn)=\chi(m)\chi(n)~$ for $m,n\in\Z$. 
     \end{itemize}
For simplicity, we  consider a {\bf primitive} Dirichlet character $\chi$ modulo $N.~$  Let $~\bar{\chi}(n)=\overline{\chi(n)}~$ and  the \emph{Gauss sum} $~\tau(\chi)~$ be defined by the formula
\[
\tau(\chi)=\sum_{n ~\textit{mod}~ N} \chi(n)e^{2\pi i \frac{n}{N}}.
\]
It is well known that 

\begin{equation}\label{character}
    \chi(n)=\frac{\chi(-1)\tau(\chi)}{N}\sum_{m ~\textit{mod}~ N} \overline{\chi(m)}e^{2\pi i \frac{mn}{N}} \qquad \textit{and}\qquad |\tau(\chi)|=N^{\frac{1}{2}}.
\end{equation}
Note  that the right-hand side of  (\ref{character}) is defined when $n$ is an arbitrary real number.

Let $m_{s,\chi}(x)$ be defined in (\ref{twisted multiplier}), since   the function $m_{s,\chi}(x)$ is in $L^2[0,1]$. Thus, the series is Abel--Gauss summable almost everywhere; 
hence,
\[
m_{s,\chi}(x)=\lim _{\epsilon\rightarrow 0} \sum_{n=1}^\infty \chi(n)n^{-s} e^{-2\pi i n^2x}e^{-\pi n^2\epsilon}.
\]
Next, we consider the $\Theta-$function

\begin{equation}\label{Theta func }
    S_y(x)=\sum_{n=-\infty}^\infty \chi(n)e^{-\pi n^2(y+2 i x)}, 
\end{equation}
Note that

\begin{equation}\label{multiplier trans}
   m_{s,\chi}(x)=C_{\alpha}\int_0^1 S_y(x)y^{-1+\frac{s}{2}} dy+O(1).   
\end{equation}
Now, in the definition of $S_y(x)$, we can replace the $~n$, which ranges over $\Z$, with $~n=mq+\ell~$, where $~m~$ ranges over $~\Z~$ and $~\ell~$ ranges over $~1\le \ell \le q~$. Then, $~S_y(x)~$ equals

\begin{equation}\label{S_y 1}
   \sum_{\ell=1}^qe^{2\pi i \ell^2\frac{p}{q}}\left\{\sum_{m\in\Z} \chi(mq+\ell) e^{- \pi (mq+\ell)^2y}\right\} .
\end{equation}
For  the inner sum, we use the Poisson summation formula $~\sum_{m\in \Z} f(m)=\sum_{m \in \Z} \Hat{f}(m)$, ~~with $~f(s)=\chi(sq+\ell)e^{-\pi (sq+\ell)^2y}$. Then,

\[
\begin{array}{lc}\ds
   \Hat{f}(\xi)=\int_{-\infty}^\infty \chi(sq+\ell)e^{-\pi(sq+\ell)^2}e^{-2\pi i s\cdot \xi} ds
   \\\\ \ds 
 =\frac{\chi(-1)\tau(\chi)}{qN}\sum_{k ~\textit{mod}~ N} \overline{\chi(k)}  \int_{-\infty}^\infty \chi(u)e^{-\pi u^2}e^{-2\pi i \left(\frac{u-\ell}{q}\right)\cdot \xi}e^{2\pi i \frac{ku}{N}}du
   \\\\ \ds 
    =\frac{\chi(-1)\tau(\chi)}{qy^{\frac{1}{2}}N}\sum_{k ~\textit{mod}~ N} \overline{\chi(k)}  e^{2\pi i \frac{\ell \xi}{q}} e^{-\frac{\pi }{y}\left(\frac{\xi}{q}-\frac{k}{N}\right)^2}.
   \\\\ \ds 
   \end{array}
\]
Therefore, we have
\begin{equation}\label{S_y 2}
   S_y(x)=\frac{\chi(-1)\tau(\chi)}{q(y+2 i\delta)^{\frac{1}{2}}N}\sum_{k ~\textit{mod}~ N} \overline{\chi(k)} 
\sum_{m=-\infty}^{\infty} S\left(\frac{p}{q},\frac{m}{q}\right)
 e^{-\frac{\pi}{y+2 i\delta}\left(\frac{m}{q}-\frac{k}{N}\right)^2}, 
\end{equation}
where
\[
S\left(\frac{p}{q},\frac{m}{q}\right)=\sum_{\ell=1}^qe^{2\pi  i \left(\frac{p}{q}\ell^2+\frac{m}{q}\ell\right)},\qquad x=\frac{p}{q}+\delta.
\]
It suffices to consider $T_y(x)$ defined as
\begin{equation}\label{T_y}
    T_y(x)=\frac{1}{q(y+2 i\delta)^{\frac{1}{2}}}
\sum_{m=-\infty}^{\infty} S\left(\frac{p}{q},\frac{m}{q}\right)
 e^{-\frac{\pi}{y+2 i\delta}\left(\frac{m}{q}-\frac{1}{N}\right)^2}.
\end{equation}
Now, write 
\begin{equation}\label{decomp}
   \int_0^1 T_y(x) y^{-1+\frac{s}{2}}dy=\sum_{j=0}^\infty \int_{2^{-j-1}}^{2^{-j}}T_y(x) y^{-1+\frac{s}{2}}dy. 
\end{equation}
and estimate $T_y(x)$ when $y$ is of a fixed order of magnitude and $x$ is ``sufficiently close''  to an appropriate rational $~\frac{p}{q}~$, with $~(p,q)=1,~0<\frac{p}{q}\le1$. Actually, we have the following Lemma.
\begin{lemma}
If $x=\frac{p}{q}+\delta~$, with $~q\lesssim y^{-\frac{1}{2}}~$ and $~q|\delta|\lesssim y^{\frac{1}{2}}$,~then 
\begin{equation}\label{T_y estimate}
 \begin{array}{cc}
    T_y(x) =
 S\left(\frac{p}{q},\frac{1}{N}\right)\frac{1}{q(y+2 i\delta)^{\frac{1}{2}}}+O(y^{-\frac{1}{4}}),\qquad \textit{if} \qquad N|q,
 \\\\ \ds
~~~ T_y(x)=O( y^{-\frac{1}{4}}) \qquad~~~~~~~~~~~~~~~~~~~~~~~~~~~~\textit{if} \qquad N\nmid q.
 \end{array}
\end{equation}
\end{lemma}
\begin{proof} 
It suffices to prove that
\[
O\left(\frac{1}{q(y+2 i\delta)^{\frac{1}{2}}}
\sum_{m\neq \frac{q}{N}} S\left(\frac{p}{q},\frac{m}{q}\right)
 e^{-\frac{\pi}{y+2 i\delta}\left(\frac{m}{q}-\frac{1}{N}\right)^2}\right)=O(y^{-\frac{1}{4}}).
\]
Note that $~S\left(\frac{p}{q},\frac{m}{q}\right)=O(q^{\frac{1}{2}})$;~we can write
 
\[
\begin{array}{cc}\ds
    O\left(\frac{1}{q(y+2 i\delta)^{\frac{1}{2}}}
\sum_{m\neq \frac{q}{N}} S\left(\frac{p}{q},\frac{m}{q}\right)
 e^{-\frac{\pi}{y+2 i\delta}\left(\frac{m}{q}-\frac{1}{N}\right)^2}\right)=O\left(q^{-\frac{1}{2}}(y^2+4\delta^2)^{-\frac{1}{4}}\sum_{m\neq\frac{q}{N}}e^{- \frac{\pi y}{y^2+4\delta^2} \left(\frac{m}{q}-\frac{1}{N}\right)^2 } \right)
 \\\\ \ds~~~~~~~~~~~~~~~~~~~~~~~~~~~~~~~~~~~~~~~~~~~~~~~~~~~~~~~~~~~~~~~~~~~~
 =O\left(q^{-\frac{1}{2}}(y^2+4\delta^2)^{-\frac{1}{4}}\sum_{m=1 }^\infty e^{- \frac{\pi y}{y^2+4\delta^2}   \frac{m^2}{q^2}  } \right).
\end{array}
\]
 
However,  by our assumptions we have $1~\lesssim \frac{y}{q^2(y^2+4\delta^2)}$. Moreover,
\[
\sum_{m=1}^{\infty}e^{-C m^2 u}\lesssim e^{-\bar{C}u}\lesssim u^{-\frac{1}{4}}, \qquad 1\lesssim u.
\]
So,  the error term is 
\[
O\left(q^{-\frac{1}{2}}(y^2+4\delta^2)^{-\frac{1}{4}} \cdot y^{-\frac{1}{4}} q^{\frac{1}{2}}\cdot (y^2+4\delta^2)^{\frac{1}{4}}  \right)=O(y^{-\frac{1}{4}}),
\]
and (\ref{T_y estimate}) is proved.
\end{proof}

\begin{proof}[Proof of Theorem 1]
Let us turn to (\ref{decomp}),  and  we make the same decomposition of the $x$-interval 
as in \cite{Stein-Wainger1} .  For $~y$ of the order $~2^{-j}$, we make a Farey dissection of the $x$-interval $[0,1]$. Now, we choose all fractions $p/q,~(p,q)=1,~$ with $~q\le 2^{\frac{j}{2}}$, ~and let $I^j_{p/q}$ be the corresponding interval for $p/q$. Then, $~I^j_{p/q}\subset \{x: |x-p/q|\le 1/q(2^{j/2})\}$.
Then, we can define the major arcs and minor arcs as follows:
\[
I_{p/q}^j \quad \textit{is~a~ major ~arc ~if}\quad  q\le \frac{1}{10}2^{j/2},
\]
\[
I_{p/q}^j \quad \textit{is~a~ minor ~arc ~if}\quad   \frac{1}{10}2^{j/2}\le q\le 2^{j/2}.
\]
Additionally,  we define $\tilde{I}_{p/q}$, independent of $j$, as
\[
\tilde{I}_{p/q}=\{x:|x-p/q|\le 1/10q^2\}.
\]
The  key property of $\tilde{I}_{p/q}$ is that if $q\le q'\le 2q$, the intervals $\tilde{I}_{p'/q'}~$ and $~\tilde{I}_{p/q}$ are disjointed (or identical) (see \cite{Stein-Wainger1}).

Now, we apply (\ref{T_y estimate}). If $~x~$ belongs to a major arc, this implies 
\begin{equation}\label{major estimate}
     \begin{array}{cc}
    T_y(x) =
 S\left(\frac{p}{q},\frac{1}{N}\right) \frac{1}{q(y+2 i\delta)^{\frac{1}{2}}}+O(2^{\frac{j}{4}}),\qquad \textit{if} \qquad N|q,
 \\\\ \ds
~~ T_y(x)=O(2^{\frac{j}{4}}) \qquad~~~~~~~~~~~~~~~~~~~~~~~~~~~~~~ \textit{if} \qquad N\nmid q.
 \end{array}
\end{equation}
If $~x~$ belongs to a minor arc,
then 
\begin{equation}\label{minor estimate}
     \begin{array}{cc}
 T_y(x)=O(2^{\frac{j}{4}}) .
 \end{array}
\end{equation}
This is because
\[
\frac{1}{q} S\left(\frac{p}{q},\frac{1}{N}\right) =O(q^{-\frac{1}{2}})=O(2^{-j/4}),
\]
and
\[
|y+2 i \delta|^{-\frac{1}{2}}\le y^{-\frac{1}{2}}=O(2^{\frac{j}{2}}).
\]
The contribution from all the minor arcs is therefore
\[
O\left(\sum_{j=0}^\infty \int_{2^{-j-1}}^{2^{-j}}  2^{\frac{j}{4}} y^{-1+\frac{s}{2}} dy \right) =O\left(
\sum_{j=0}^\infty  2^{\frac{j}{4}}  2^{-j(-1+\frac{s}{2})} 2^{-j} \right)=O(1). \qquad \textit{since}\qquad s>\frac{1}{2}.
\]
Next,  we sum over the major arcs. Fix $\frac{p}{q}$; then,
\[
\begin{array}{cc}
 \frac{1}{q}\left|S\left(\frac{p}{q},\frac{1}{N}\right)\right| \sum_{j=0}^\infty \int_{2^{-j-1}}^{2^{-j}} \left|y+2 i (x-\frac{p}{q})\right|^{-\frac{1}{2}}y^{-1+\frac{s}{2}}dy \cdot \chi_{I^j_{\frac{p}{q}}}(x)
 \\\\ \ds
 \le C  \frac{1}{q}\left|S\left(\frac{p}{q},\frac{1}{N}\right)\right|    \left|x-\frac{p}{q}\right|^{-\frac{1}{2}+\frac{s}{2}}  \cdot \chi_{ \tilde{I}_{\frac{p}{q}}}(x), \qquad 0<s<1.
 \end{array}
\]
Therefore, the total contribution of the major arcs is majorized  
by
\begin{equation}\label{major sum}
    \sum_{\frac{p}{q},~ N|q} q^{-\frac{1}{2}} \left|x-\frac{p}{q}\right|^{-\frac{1}{2}+\frac{s}{2}}  \cdot \chi_{ \tilde{I}_{\frac{p}{q}}}(x).
\end{equation}
Rewrite  sum (\ref{major sum}) as
\begin{equation}\label{major sum2}
    \sum_{t=0}^{\infty}\sum_{\substack{p/q,~ N|q \\ 2^t\le q<2^{t+1} }} q^{-\frac{1}{2}} \left|x-\frac{p}{q}\right|^{-\frac{1}{2}+\frac{s}{2}}  \cdot \chi_{ \tilde{I}_{\frac{p}{q}}}(x).
\end{equation}
Note that there are at most $~O(N^{-1}2^{2t})$ disjointed intervals for $~2^s\le q<2^{t+1}~$. Moreover, $~|x-p/q|^{-1/2+s/2}~$ is uniformly of weak-type $ L^{\frac{2}{1-s}}$.  Thus, applying Lemma 
 one in \cite{Stein-Wainger1} , then 
 
\[
\sum_{\substack{p/q,~ N|q \\ 2^t\le q<2^{t+1}}} q^{-\frac{1}{2}} \left|x-\frac{p}{q}\right|^{-\frac{1}{2}+\frac{s}{2}}  \cdot \chi_{ \tilde{I}_{\frac{p}{q}}}(x).
\]
 has a weak-type $L^{\frac{2}{1-s}}$ norm bounded by 
 \[
 q^{-\frac{1}{2}}(N^{-1}2^{2t})^{\frac{1-s}{2}}=O(2^{-\frac{t}{2}}N^{-\frac{1-s}{2}}2^{t(1-s)})=O(N^{-\frac{1-s}{2}}2^{t(\frac{1}{2}-s)}),
 \]
and the sum in (\ref{major sum2}) converges if $s>\frac{1}{2}$. This means that $m_\alpha $is of weak-type $L^{\frac{2}{1-s}}[0,1]$. The proof is complete.
\end{proof}

\section{Multiplier Related to Euler's Identity }\label{sec3}
Let $\mathbb{H}$ be the Poincar\'e upper half plane consisting of $z=x+ i y$ where $x,y\in \R$ and $y>0$. 
Suppose $f$ is  defined on $\mathbb{H}$ and  has a Fourier expansion
\[
f(z)=\sum_{n=0}^\infty a_ne^{2\pi i nz}.
\]
We  consider
the multiplier
\begin{equation}\label{multiplier 2}
    m_{s,f}(x)=\sum_{n=1}^\infty \frac{ a_n e^{-2\pi i nx}}{n^s} .
\end{equation}
For  $s>0$, applying the well-known formula
\[
\int_0^\infty e^{-2\pi ny}y^s\frac{dy}{y}=\frac{\Gamma(s)}{(2\pi n)^s},
\]
we have
\begin{equation}\label{m est2}
  m_{s,f}(x)=c_s\int_0^1 f(-x+ i y) y^s \frac{dy}{y}+O(1),  
\end{equation}
where $c_s>0$ is a constant that only depends on $s.$

 Now, let us consider the case
\[
f(z)=\prod_{n=1}^\infty (1-e^{2\pi inz})=\sum_{n=0}^\infty a_n e^{2\pi  i  nz},\quad z\in \mathbb{H},
\]
and recall Euler's identity:
\[
f(z)=\prod_{n=1}^\infty (1-e^{2\pi i nz})=\sum_{n=-\infty}^{\infty} (-1)^ne^{2\pi i\frac{3n^2+n}{2}z}
\]
Hence, we can write $f(z)=f_1(z)+f_2(z)$,
where
\[
f_1(z)=\sum_{n\in \Z}e^{2\pi i (6n^2+n)z},\qquad f_2(z)=-\sum_{n\in \Z}e^{2\pi i (6n^2+7n+2)z}.
\]
Our analysis will then proceed by setting 

\begin{equation}\label{discrete2}
   \int_0^1 f(-x + i y)y^s\frac{dy}{y}=\sum_{j=0}^\infty \int_{2^{-j-1}}^{2^{-j}} f(-x+ i y)y^s\frac{dy}{y}, 
\end{equation}
and estimating $f(-x+ i y)$ when $y$ is of a fixed order of magnitude and $x$ is "sufficiently close " to an appropriate rational $p/q$, with ~$(p,q)=1,~0<p/q\le 1$.

\begin{lemma}
    Let $x\in [0,1]$.~If 
    \[
   x=p/q+\delta, \qquad q\lesssim y^{-\frac{1}{2}},\qquad q|\delta|\lesssim y^{\frac{1}{2}},
    \]
    then 
\begin{equation}\label{fz2}
  f_1(-x+ i y)=\frac{1}{\sqrt{12}q(y+ i\delta)^{1/2}}e^{\frac{\pi (y+ i \delta)}{12}} S(p/q) +O(y^{-1/4}),  
\end{equation}
where
\[
S(p/q)=\sum_{\ell=1}^q e^{-2\pi  i \frac{p}{q}(6\ell^2+\ell)}.
\]

\end{lemma}

\begin{proof}
First, consider $f_1(-\frac{p}{q}+ i y)$. Write $\sum_{n\in\Z}=\sum_{\ell=1}^q\sum_{m\in\Z}$ and $n=mq+\ell$. Then, $f_1(-\frac{p}{q}+ i y)$ equals

\begin{equation}\label{decom2}
    \sum_{\ell=1}^q e^{-2\pi  i\frac{p}{q} (6\ell^2+\ell) } \left\{
\sum_{m\in \Z} e^{-2\pi y[6(mq+\ell)^2+mq+\ell]} \right\}
\end{equation}
For the inner sum,  set $h(x)=e^{-2 \pi y[6(xq+\ell)^2+xq+\ell]}$. Then, its Fourier transform is 
\[
\begin{array}{cc}\ds
\Hat{h}(\xi)=\int_{-\infty}^{\infty} e^{-2\pi y[6(xq+\ell)^2+xq+\ell]}e^{-2\pi i x\cdot \xi} dx  \\\\ \ds
=\frac{1}{q}\int_{-\infty}^\infty e^{-2\pi y (6u^2+u)}e^{-2\pi i \left(\frac{u-\ell}{q}\right)\xi}du \\\\ \ds
=\frac{1}{\sqrt{12}qy^{1/2}}e^{\frac{\pi y}{12}} e^{2\pi i \left(\frac{\ell \xi}{q}+\frac{\xi}{12q}\right)} e^{-\frac{\pi \xi^2}{12 q^2y}}.

\end{array}
\]
 Now, using the Poisson summation formula, we have
\[
f_1(-\frac{p}{q}+ i y)=\frac{1}{\sqrt{12}qy^{1/2}}e^{\frac{\pi y}{12}} \sum_{m\in\Z} S(p/q,m/q) e^{2\pi i \frac{m}{12 q}} e^{-\frac{\pi m^2}{12 q^2 y}},
\]
Therefore, let $y\rightarrow y+i\delta$ , which  yields

\begin{equation}\label{f1 decom}
    f_1(-x+ i y)=\frac{1}{\sqrt{12}q(y+ i\delta)^{1/2}}e^{\frac{\pi (y+ i \delta)}{12}} \sum_{m\in\Z} S(p/q,m/q) e^{2\pi i \frac{m}{12 q}} e^{-\frac{\pi m^2}{12 q^2 (y+i \delta)}}.
\end{equation}
Now, set $S(p/q)=S(p/q,0)$.~
From (\ref{f1 decom}), we see, upon isolating the term $m=0$,
\[
\begin{array}{cc}\ds
    f_1(-x+ i y)=\frac{1}{\sqrt{12}q(y+ i\delta)^{1/2}}e^{\frac{\pi (y+ i \delta)}{12}} S(p/q) \\\\ \ds
    +O(q^{-1} (y^2+\delta^2)^{-1/4}\cdot q^{\frac{1}{2}}\cdot \sum_{m=1}^\infty e^{-\pi m^2y/(12 q^2(y^2+\delta^2))} ).
\end{array}
\]
Hence, similar to the proof of  Lemma 1, (\ref{fz2}) is proved.
\end{proof}

\begin{theorem}
    When $1/4<s<1/2$, ~~$m_{s,f}$ belongs to weak-type $L^{2/(1-2s)}[0,1]$.
\end{theorem}
\begin{proof}
It suffices to prove that $m_{s,f_1}\in L^{2/(1-2s),\infty}[0,1]$.   Let us turn to (\ref{m est2}),  and we make the same decomposition of the $x$-interval as in \cite{Stein-Wainger1} .

Now, we apply (\ref{fz2}). If $~x~$ belongs to a major arc, this implies that

\begin{equation}\label{major estimate2}
     \begin{array}{cc}
f_1(-x+ i y)=\frac{1}{\sqrt{12}q(y+ i\delta)^{1/2}}e^{\frac{\pi (y+ i \delta)}{12}} S(p/q) +O(2^{j/4}),
 \end{array}
\end{equation}
If $x$ belongs a minor arc,
then

\begin{equation}\label{minor estimate2}
     \begin{array}{cc}
 f_1(-x+i y)=O(2^{j/4}),
 \end{array}
\end{equation}
This is because
\[
\frac{1}{q} S\left(\frac{p}{q}\right) =O(q^{-\frac{1}{2}})=O(2^{-j/4}),
\]
and
\[
|y+ i \delta|^{-\frac{1}{2}}\le y^{-\frac{1}{2}}=O(2^{\frac{j}{2}}).
\]
The contribution from all the minor arcs is therefore
\[
O\left(\sum_{j=0}^\infty \int_{2^{-j-1}}^{2^{-j}}  2^{\frac{j}{4}} y^s \frac{dy}{y} \right) =O\left(
\sum_{j=0}^\infty  2^{\frac{j}{4}}  2^{-j(s-1)} 2^{-j} \right)=O(1),\qquad \textit{since}\qquad s>\frac{1}{4}.
\]
Next, we sum over the major arcs. Fix $\frac{p}{q}$. Then,
\[
\begin{array}{cc}
 \frac{1}{q}\left|S\left(\frac{p}{q}\right)\right| \sum_{j=0}^\infty \int_{2^{-j-1}}^{2^{-j}} \left|y+ i (x-\frac{p}{q})\right|^{-\frac{1}{2}}y^s\frac{dy}{y} \cdot \chi_{I^j_{\frac{p}{q}}}(x)
 \\\\ \ds
 \le  C  \frac{1}{q}\left|S\left(\frac{p}{q}\right)\right|    \left|x-\frac{p}{q}\right|^{s-\frac{1}{2}} \cdot \chi_{ \tilde{I}_{\frac{p}{q}}}(x), 
 \end{array}
\]
because
\[
\int_0^1 \left|y+ i \delta\right|^{-\frac{1}{2}}y^s\frac{dy}{y} \le \int_0^\infty \left|y+ i \delta\right|^{-\frac{1}{2}}y^s\frac{dy}{y} =C |\delta|^{s-1/2},
\]
as long as $0<s<1/2$.

Therefore, the total contribution of the major arcs is majorized by
\begin{equation}\label{major sum3}
   \sum_{p/q} q^{-\frac{1}{2}} \left|x-\frac{p}{q}\right|^{s-1/2 }  \cdot \chi_{ \tilde{I}_{\frac{p}{q}}}(x). 
\end{equation}
Rewrite sum (\ref{major sum3}) as
\[
\sum_{t=0}^{\infty}\sum_{\substack{p/q,~ \\  2^t\le q<2^{t+1}}} q^{-\frac{1}{2}} \left|x-\frac{p}{q}\right|^{s-1/2 }  \cdot \chi_{ \tilde{I}_{\frac{p}{q}}}(x).
\]
Note that there are at most $~O(2^{2t})$ disjointed intervals for $2^t\le q<2^{t+1}$. Moreover, $~|x-p/q|^{s-1/2 }~$ is uniformly of weak-type $L^{\frac{2}{1-2s } }[0,1]$.  Thus, applying Lemma 
 one in \cite{Stein-Wainger1} , then 
 
\[
\sum_{\substack{ p/q, \\ 2^t\le q<2^{t+1}}} q^{-\frac{1}{2}} \left|x-\frac{p}{q}\right|^{s-1/2 }  \cdot \chi_{ \tilde{I}_{\frac{p}{q}}}(x).
\]
 has a weak $L^{\frac{2}{1-2s }}$ norm bounded by 
 \[
 q^{-\frac{1}{2}} 2^{2t(1/2-s) }=O(2^{-\frac{t}{2}} 2^{t(1-2s)})=O( 2^{t(1/2-2s)}),
 \]
and the sum in (\ref{major sum3}) converges if $s>1/4$. This means that $m_{s,f_1}\in L^{2/(1-2s),\infty}[0,1]$. Therefore, $m_{s,f}$ belongs to weak-type $L^{2/(1-2s)}[0,1]$.
\end{proof}

\subsection*{The Corresponding Discrete Fractional Integral}

Let $g$ be a function of $\Z$. The twisting discrete fractional operator $\I_{s,f}$ is  defined as 
\[
\I_{s,f} g(m)=\sum_{n=1}^{\infty} \frac{a_n}{n^s}g(m-n) ,\quad m\in\Z^1
\]
and has acting functions defined on $\Z^1$. If $g(n)=1,~n\in \Z$,~ then ~$\I_{s,f}1=\sum_{n=1}^\infty a_n n^{-s}$ is usually called the $L-$ function of $f.~$  On the other hand, if~ $a_n\equiv 1,n=1,2,\dots,$ ~then we can write 
\[
\I_sg(m)=\sum_{n=1}^{\infty} \frac{g(m-n)}{n^s} ,\quad m\in\Z^1,
\]
Stein and Wainger \cite{Stein-Wainger1} proved the following theorem.

\begin{theorem}
  For  $0<s<1$, then
\[
||\I_s f||_{\ell^q(\Z)}\le C||f||_{\ell^p(\Z)},
\]
if  
\[
\frac{1}{q}\le \frac{1}{p}-1+s,\qquad 1<p<q<\infty.
\]  
\end{theorem}
Then it is easy to see that $a_n=\pm 1,0$ for all $n\in \mathbb{N}$. Then, if $1/q\le 1/p-1+s,~1<p<q<\infty$,

\[
||\I_{s,f} g||_{\ell^q(\Z)}\le C||g||_{\ell^p(\Z)}.
\]
However, if we take the cancellation property of $a_n$ into consideration we have a better result in the range of $1/4<s<1/2$.

\begin{theorem}
  For  $1/4<s<1/2$, then
\[
||\I_{s,f}g||_{\ell^q(\Z)}\le C||g||_{\ell^p(\Z)},
\]
if  
\[
\frac{1}{q}\le \frac{1}{p}-\frac{1}{2}+s,\qquad 1<p\le 2\le q<\infty.
\]  
\end{theorem}
In order to obtain the the desired $\ell^p\rightarrow \ell^q$ inequalities, we need a “folk” Lemma due to Stein and Wainger \cite{Stein-Wainger1} concerning a convolution operator $\I$ with multiplier $m$, viz.,

\[
\widehat{\I f}(x)=m(x)\Hat{f}(x).
\]

\begin{lemma}
    Assume $m(x)$ is of weak-type $ L^{r}[0,1]$. Then, $\I$ is bounded from $\ell^p(\Z^1)$ to $\ell^q(\Z^1)$ if 
    \[
    \frac{1}{p}-\frac{1}{q}=\frac{1}{r},\qquad 1<p\le 2\le q<\infty.
    \]
\end{lemma}
\begin{proof}
First, assume that $q=2$, so that $\frac{1}{p}=\frac{1}{2}+\frac{1}{r}$. Then, for $f\in \ell^p(\Z^1)$, using Paley's version of the Hausdorff--Young inequality, $~\Hat{f}\in L^{p',p}[0,1]$, where $\frac{1}{p}+\frac{1}{p'}=1$. Therefore, by the multiplicative property of Lorentz spaces, 
\[
\widehat{\I f}(x)=m(x)\Hat{f}(x)\in L^{r,\infty}\cdot L^{p',p}\subset L^{p_0,q_0},
\]
where $\frac{1}{p_0}=\frac{1}{p'}+\frac{1}{2}=\frac{1}{2},~\frac{1}{q_0}=\frac{1}{p}+\frac{1}{\infty}=\frac{1}{p}~$. Therefore, $~\Hat{\I f}\in L^{2,p}\subset L^{2,2}=L^2~$ since $p\le 2$ by assumption. Hence, $~\I f\in \ell^2(\Z^1)~$ and so $\I$ maps $\ell^p(\Z^1)$ to $~\ell^2(\Z^1)~$. The case when $~p=2~$ and $~\frac{1}{2}-\frac{1}{r}=\frac{1}{q}~$ follows by considering the adjoint operator of $\I$, and the Lemma then follows by interpolation between the two resulting bounds for $~\I$.
\end{proof}
\vspace{6pt}
Now, { {Theorem 4}} is just an immediate corollary of { {Lemma 3}} and { {Theorem 2}} and note that  $\ell^p \subset \ell^q$ if $p\le q$.

\section{Multipliers Related to Imaginary Quadratic Fields}\label{sec4}
Now, consider the imaginary quadratic field $K=\Q(\sqrt{D})$ of discriminant $D<0$.
~Let $\mathcal{O}\subset K$ denote the ring of integers.~ Let $I$ be the group of fractional ideals $\neq 0$,
\[
I=\left\{\frac{\mathfrak{a}_1}{\mathfrak{a}_2} : ~\mathfrak{a}_1, \mathfrak{a}_2\subset \mathcal{O}, \mathfrak{a}_1\mathfrak{a}_2\neq 0\right\}
\]
and  $P\subset I$ the subgroup of principal ideals
\[
P=\{ (a)=a\mathcal{O}:  ~a\in K^*\}.
\]
Then, $\mathcal{H}=I/P$ is the class group, and the class number of $K$ is $~h=[I:P]<\infty$. For more detailed background of imaginary 
quadratic fields, see \cite{Iwaniec}.
 Define the fractional integral associated with the imaginary quadratic field by 

\begin{equation}\label{op3}
   \I_{s,K}f(m)=\sum_{0\neq \mathfrak{a} \subset \mathcal{O}}\frac{f(m-N(\mathfrak{a}))}{N(\mathfrak{a})^s},\qquad m\in \Z, 
\end{equation}
where $\mathfrak{a}$ ranges over non-zero integral ideals and $~N: I\rightarrow \Q^*$ is the norm.  In the case of integral ideals, $N(\mathfrak{a})=\#(\mathcal{O}/\mathfrak{a})$.  

The corresponding multiplier is 
\[
m_{s,K}(x)=\sum_{0\neq \mathfrak{a} \subset \mathcal{O}}\frac{ e^{-2\pi i x N(\mathfrak{a})}}{N(\mathfrak{a})^s} ,
\]
that is, 
\[
\widehat{\I_{s,K} f}(x) =m_{s,K}(x)\Hat{f}(x),\qquad \Hat{f}(x)=\sum_{n\in \Z}f(n)e^{-2\pi i n x}.
\]
Our main result is stated as below.
\begin{theorem}
    For  $\frac{1}{2}<s<1$, then ~$m_{s,K}$ belongs to weak-type ~$L^{1/(1-s)}[0,1]$.
\end{theorem}

\begin{proof}
Let $w$ be the number of units of $K.$ For every class $\mathcal{A} \in \mathcal{H}$, we introduce the corresponding multiplier: 

\begin{equation}\label{multiplier A}
  m_{s,\mathcal{A}}(x)=\sum_{0\neq\mathfrak{a} \in \mathcal{A}} \frac{e^{-2\pi  i x N(\mathfrak{a}) }}{N(\mathfrak{a})^s}  
\end{equation}
Every class $\mathcal{A} $ contains an integral primitive ideal, i.e., every class is not divisible by a rational integer $>1$. 
Every primitive ideal can be written as
\[
\mathfrak{a}=\left[a,\frac{b+\sqrt{D}}{2}\right]~~\textit{with} ~~a>0,~b^2-4ac=D, ~(a,b,c)=1.
\]
The above notation means $\mathfrak{a}$ is a free $\Z$-module, 
\[
\mathfrak{a}=a\Z+\frac{b+\sqrt{D}}{2}\Z.
\]
Note that $\frac{b+\sqrt{D}}{2}\in \mathcal{O}$ and $a=N(\mathfrak{a})$; with the generators of $\mathfrak{a}$, we associate the quadratic form
\[
\varphi_A(x)=ax_1^2+bx_1x_2+cx_2^2=\frac{1}{2}A[x]
\]
where 
\[
A=\left(
\begin{array}{cc}
   2a  & b  \\
    b & 2c
\end{array}
\right).
\]

This establishes a one-to-one correspondence between the ideal classes $\mathcal{A}\in \mathcal{H}$ and the equivalence classes of primitive binary quadratic forms $\varphi_A$ of discriminant $D.$ We choose $\sqrt{D}=i\sqrt{|D|}$ so that
\[
z_{\mathfrak{a}}=\frac{b+\sqrt{D}}{2a}\in \mathbb{H}.
\]
Then, the inverse ideal $\mathfrak{a}^{-1}$ is a free $\Z$-module generated by one and $\bar z_{\mathfrak{a}}$:
\[
\mathfrak{a}^{-1}=[1,\bar z_{\mathfrak{a}}]=\Z+\frac{b-\sqrt{D}}{2a}\Z.
\]
Now, given a class $\mathcal{A}$ which contains $\mathfrak{a},$ we can write
\[
m_{s,\mathcal{A}}(x)=\sum_{0\neq\mathfrak{b} \sim\mathfrak{a} \in\mathcal{A}} \frac{e^{-2\pi  i x N(\mathfrak{a}) }}{N(\mathfrak{a})^s}
\]
Here, the equivalence $\mathfrak{b} \sim \mathfrak{a}$ means $\mathfrak{b}=(\alpha)\mathfrak{a}$ with $\alpha\in \mathfrak{a}^{-1},~\alpha\neq 0,$~i.e., $\alpha=m+n\bar z_{\mathfrak{a}}$ ~with $m,n\in\Z,$ ~not both zero.
As $m,n$ range over the integers, every ideal $\mathfrak{b} \sim \mathfrak{a}$ is covered exactly $w$ times. Moreover, we have $N(\mathfrak{b})=|\alpha|^2 a=am^2+bmn+cn^2$; hence,
\[
m_{s,\mathcal{A}}(x)=\sum_{0\neq\mathfrak{a} \in \mathcal{A}} \frac{e^{-2\pi  i x N(\mathfrak{a}) }}{N(\mathfrak{a})^s}=\frac{1}{w}\sum_{(0,0)\neq (m,n)\in \Z^2} \frac{e^{-2\pi  i x \varphi_A(m,n) }}{\varphi_A(m,n)^s}
\]
On the other hand, we have
\[
m_{s,K}(x)=\sum_{\mathcal{A}\in \mathcal{H}} m_{s,\mathcal{A}}(x).
\]
Lillian Pierce \cite{Pierce Lillian 3} showed that $m_{s,\mathcal{A}}\in  L^{1/(1-s)}[0,1]$ and therefore  $m_{s,K}\in L^{1/(1-s)}[0,1]$ .
\end{proof}

\vspace{-6pt}
\section{Discrete Analogue of Stein--Weiss Inequality on Product Space}\label{sec5}
In their study of fractional integrals, Stein and Weiss \cite{Stein-Weiss} considered the operator $\I_{\alpha}$, acting on functions on $\R^\NN$, given by

\begin{equation}\label{frac integ op}
    \I_{\alpha}f(x)=\int_{\R^\NN}f(y)\left(\frac{1}{|x-y|}\right)^{\NN-\alpha}dy.
\end{equation}
If we let $~\omega(x)=|x|^{-\gamma},~\sigma(x)=|x|^{\delta}$, they proved under some conditions of $\gamma$ and$~\delta$, that the following weighted norm inequality holds:

\begin{equation}\label{weighted norm}
   ||\omega\I_\alpha f ||_{L^q(\R^\NN)}\le C ||f\sigma||_{L^p(\R^\NN)}, 
\end{equation}
which is now known as the Stein--Weiss inequality. When $\gamma=\delta=0~, ~1/q=1/p-\alpha/\NN$,  ~ (\ref{weighted norm}) is the famous {\bf Hardy--Littlewood--Sobolev} inequality, for more details, see \cite{Stein, Hardy-Littlewood, Sobolev}.

In 2021, Wang \cite{Wang}~ extended the Stein--Weiss inequality to the so-called product space case. Now, consider $\R^\N$  as a product space by writing $\R^\NN=\R^{\NN_1}\times\R^{\NN_2}\times\cdots\times \R^{\NN_k},~k\ge 2. $
Let 
\[
0<\alpha_i<\NN_i,\qquad i=1,2,\dots,k  \quad \textit{and} \quad \alpha=\alpha_1+\alpha_2+\cdots+\alpha_k.
\]
Wang actually studied the weighted norm inequality of the so-called \emph{strong fractional integral operator} $\mathcal{J}_{\alpha}~$, defined by
\begin{equation}\label{str frac integ op}
   \mathcal{J}_{\alpha}f(x)=\int_{\R^\NN}f(y)\prod_{i=1}^k \left(\frac{1}{|x_i-y_i|}\right)^{\NN_i-\alpha_i}dy. 
\end{equation}

\begin{theorem}[\cite{Wang}]
Let $0<\alpha<\NN, ~\gamma,\delta\in \R, ~1<p\le q<\infty $. Then,
    \begin{equation}\label{Wang ineq}
           ||\omega\mathcal{J}_{\alpha}f||_{L^q(\R^\NN)}\le C||f\sigma||_{L^p(\R^\NN)}, 
    \end{equation}
is equivalent to 

\begin{equation}\label{constrain1}
       \gamma<\frac{\NN}{q},\quad \delta<\NN\frac{p-1}{p},\quad \gamma+\delta\ge 0,\qquad
  \frac{1}{q}=\frac{1}{p}+\frac{\gamma+\delta-\alpha}{\NN}.
\end{equation}
For $\gamma\ge 0,~\delta\le 0,$
\begin{equation}\label{constrain3}
     \alpha_i-\frac{\NN_i}{p}<\delta,\qquad i=1,2,\dots,k.
\end{equation}
    For $\gamma\le 0,~\delta\ge 0,$
    \begin{equation}\label{constrain4}
         \alpha_i-\NN_i\frac{q-1}{q}<\gamma,\qquad i=1,2,\dots,k.
    \end{equation}
For $\gamma> 0,~\delta> 0,$
  \begin{equation}\label{constrain5}
       \sum_{ \substack{i\in\{1,2,\dots,k \} \\~\alpha_i\ge \NN_i/p}}\left(\alpha_i-\frac{\NN_i}{p}\right)<\delta,\qquad 
  \sum_{ \substack{i\in\{1,2,\dots,k \} \\  \alpha_i\ge \NN_i(q-1)/q}}\left(\alpha_i-\NN_i\left(\frac{q-1}{q}\right)\right)<\gamma,\qquad 
  \end{equation}
\end{theorem}
It is natural to ask whether the discrete analogue of {{Theorem 6}}  holds.  Considering the discrete operator $\T_{\alpha,\gamma,\delta}$, defined as follows,
\begin{equation}\label{op1}
     \T_{\alpha,\gamma,\delta}f(n)=\sum_{\substack{m_1\in \Z^{\NN_1} \\ m_1\neq n_1}} \sum_{\substack{m_2\in \Z^{\NN_2} \\  m_2\neq n_2}}\cdots\sum_{\substack{m_k\in \Z^{\NN_k} \\  m_k\neq n_k}}\frac{f(m)}{|n|^{\gamma}|m|^\delta} \prod_{i=1}^k \left(\frac{1}{|n_i-m_i|}\right)^{\NN_i-\alpha_i},\qquad |m||n|\neq 0.
\end{equation}
where 
\[
m=(m_1,m_2,\dots,m_k),\quad n=(n_1,n_2,\dots.n_k)\in \Z^{\NN_1} \times\Z^{\NN_2}\times\cdots\times\Z^{\NN_k},
\]
the following Theorem is expected. 
\begin{theorem}
Under the conditions of  (\ref{constrain1})--(\ref{constrain5}), we have 
    \begin{equation}\label{discrete Stein-Weiss prod}
        ||\T_{\alpha,\gamma,\delta}f||_{\ell^q(\Z^\NN)}\le  C||f||_{\ell^p(\Z^\NN)}.
\end{equation}
\end{theorem}

\begin{proof} We can assume that $f\ge 0$. Let $~Q=Q_1\times Q_2\times\cdots\times Q_k\subset \R^{\NN_1}\times\R^{\NN_2}\times \cdots\times\R^{\NN_k} $, where 

\[
Q_i=\left\{x_i=(x_i^1,x_i^2,\dots,x_i^{\NN_i}):~-\frac{1}{2}<x_i^j\le \frac{1}{2},~~j=1,2,\dots,\NN_i\right \},\quad i=1,2,\dots, k.
\]
Define $F$, associated with $f$, as

\begin{equation}\label{F asso f}
   F(x)=f(n),\qquad x\in Q+n, 
\end{equation}
 $Q$ is the fundamental cube in $\R^\NN$. 
Since 
\[
\R^\NN=\bigcup_{n\in\Z^\NN}(Q+n),
\]
$F$ is well defined. Note that for the appropriate constant $C$,

\begin{equation}\label{ineq1}
    \begin{array}{cc}\ds
 |n_i-m_i|^{-\NN_i+\alpha_i}\le C|n_i+u_i-(m_i+v_i)|^{-\NN_i+\alpha_i}, \quad i=1,2,\dots,k,
 \\\\ \ds
 |n|^{-\gamma}\le C|n+u|^{-\gamma},\quad |m|^{-\delta}\le  C|m+v|^{-\delta},
 \end{array}
\end{equation}
where $u_i,v_i\in Q_i,~i=1,2,\dots,k~$ and $u=(u_1,u_2,\dots,u_k),~v=(v_1,v_2,\dots,v_k).~$ Therefore, we have

\begin{equation}\label{ineq2}
    \begin{array}{lc}\ds
\sum_{ \substack{m_1\in \Z^{\NN_1}\\  m_1\neq n_1}} \sum_{\substack{m_2\in \Z^{\NN_2}\\  m_2\neq n_2}}\cdots\sum_{ \substack{m_k\in \Z^{\NN_k} \\  m_k\neq n_k}}\frac{f(m)}{|n|^{\gamma}|m|^\delta} \prod_{i=1}^k \left(\frac{1}{|n_i-m_i|}\right)^{\NN_i-\alpha_i} 
\\\\ \ds
 \le C \sum\limits_{\substack{m_1\in \Z^{\NN_1} \\  m_1\neq n_1}} \sum\limits_{\substack{m_2\in \Z^{\NN_2} \\  m_2\neq n_2}}\cdots\sum\limits_{\substack{m_k\in \Z^{\NN_k} \\  m_k\neq n_k}}
\left\{\prod_{i=1}^k \int_{v_i\in Q_i}\right\} 
 \frac{f(m)}{|n+u|^{\gamma}|m+v|^\delta}  \prod_{i=1}^k \left(\frac{1}{|n_i+u_i-m_i-v_i|}\right)^{\NN_i-\alpha_i}dv_i
 \\\\ \ds 
 \le C \sum\limits_{ \substack{m_1\in \Z^{\NN_1} \\ m_1\neq n_1}} \sum\limits_{ \substack{m_2\in \Z^{\NN_2}\\  m_2\neq n_2}}\cdots\sum\limits_{\substack{m_k\in \Z^{\NN_k} \\  m_k\neq n_k}}
 \left\{\prod_{i=1}^k \int_{v_i\in Q_i}\right\} 
 \frac{F(m+v)}{|x|^{\gamma}|m+v|^\delta}  \prod_{i=1}^k \left(\frac{1}{|x_i-m_i-v_i|}\right)^{\NN_i-\alpha_i} dv_i
 \\ \ds 
 ~~~~~~~~~~~~~~~~~~~~~~~~~~~~~~~~~~(x=n+u)
 \\ \ds 
 \le C \sum\limits_{\substack{m_1\in \Z^{\NN_1} \\ m_1\neq n_1}} \sum\limits_{\substack{m_2\in \Z^{\NN_2} \\  m_2\neq n_2}}\cdots\sum\limits_{ \substack{m_k\in \Z^{\NN_k} \\
 m_k\neq n_k}}
  \left\{\prod_{i=1}^k \int_{y_i\in Q_i+m_i}\right\} 
  \frac{F(y)}{|x|^{\gamma}|y|^\delta} \prod_{i=1}^k \left(\frac{1}{|x_i-y_i|}\right)^{\NN_i-\alpha_i}dy_i
  \\ \ds~~~~~~~~~~~~~~~~~~~~~~~~~~~~~
  \quad (y=m+v)
 \\ \ds  
 \le  C \int_{\R^\NN} \frac{F(y)}{|x|^{\gamma} |y|^\delta} \prod_{i=1}^k \left(\frac{1}{|x_i-y_i|}\right)^{\NN_i-\alpha_i}  dy =C \T^*_{\alpha,\gamma,\delta}F(x), \quad x\in Q+n.
 \end{array}
\end{equation}
 
Now, by (\ref{ineq2}) and applying {{Theorem 6}}, we have
\begin{equation}\label{ineq3}
    \begin{array}{cc}\ds
   ||\T_{\alpha,\gamma,\lambda}f||_{\ell^q(\Z^\NN)}=\left(\sum_{n\neq 0} \T_{\alpha,\gamma,\lambda}f(n)^q \right)^{\frac{1}{q}} \le C \left(\sum_{n\neq 0}\int_{x\in Q+n}\T^*_{\alpha,\gamma,\lambda}F(x)^qdx \right)^{\frac{1}{q}} \ds
   \\\\
   \le C ||\T^*_{\lambda,\alpha,\beta}F||_{L^q(\R^\NN)}\le C||F||_{L^p(\R^\NN)}=C||f||_{\ell^p(\Z^\NN)},
   \end{array}
\end{equation}
Under the same conditions as (\ref{constrain1})--(\ref{constrain5}), note that $\ell^{p_1}(\Z^\NN)\subset 
 \ell^{p_2}(\Z^\NN)$ if $~p_1\le p_2$, We can actually improve (\ref{constrain1}) to 
 \[
   \frac{1}{q}\le \frac{1}{p}+\frac{\gamma+\delta-\alpha}{\NN}.
   \]
 The proof is completed.
\end{proof}

\addcontentsline{toc}{section}{Acknowledgements}		
	

\begin{thebibliography}{999}
     \bibitem[R28]{R28}
Riesz, M.    Sur les fonctions conjuguées. {\em Math. Z.} {\bf 1928}, {\em 27}, 218--244. 

  
 \bibitem[SW00]{Stein-Wainger1}
Stein, E.;  Wainger, S.   Discrete analogues in harmonic analysis. {II}. Fractional
              integration. {\em J. Anal. Math.} {\bf 2000}, {\em 80}, 335--355.
              
       \bibitem[W21]{W21}
Wainger, S.    An introduction to the circle method of {H}ardy, {L}ittlewood,
              and {R}amanujan. {\em J. Geom. Anal.} {\bf 2021}, {\em 31}, 9113--9130. 

\bibitem[P19]{Pierce Lillian 3}
Pierce, L.  Discrete Analogues in Harmonic Analysis. Ph.D. Thesis, Princeton University,  NJ, USA, 2019.

  
  \bibitem[SW02]{Stein-Wainger2}
Stein, E.;  Wainger, S. Two discrete fractional integral operators revisited. {\em J. Anal. Math.} {\bf 2002}, {\em 87}, 451--479.


\bibitem[O01]{Oberlin}
Oberlin, D. Two discrete fractional integrals. {\em Math. Res. Lett.} {\bf 2001}, {\em 8}, 1--6.
  
   
    \bibitem[P21]{Pierce Lillian 1}
Pierece, L.     Discrete fractional {R}adon transforms and quadratic forms. {\em Duke Math. J.} {\bf 2012}, {\em 161}, 69--109. 
 
 
    \bibitem[P11]{Pierce Lillian 2}
Pierece, L.     On discrete fractional integral operators and mean values of
              {W}eyl sums. {\em Bull. Lond. Math. Soc.} {\bf 2011}, {\em 43}, 597--612. 
   
   
                
     \bibitem[BW23]{B23}
Bourgain, J.;  Stein,  E.;  Wright,  J.   On a multi-parameter variant of the Bellow–Furstenberg problem. {\em Forum Math. Pi} {\bf 2023}, {\em 11}, 1--64. 

                     \bibitem[BMSW19]{B19}
Bourgain, J.;  Mirek, M.;  Stein,  E.;     Wr\'{o}bel, B.  Dimension-free estimates for discrete {H}ardy-{L}ittlewood
              averaging operators over the cubes in {$\Bbb Z^d$}. {\em Amer. J. Math. } {\bf 2019}, {\em 141}, 857--905. 
 
      \bibitem[MST19]{M19}
Mirek, M.;  Stein,  E.;  Trojan,  B.   {$\ell^p(\Bbb{Z}^d)$}-estimates for discrete operators of
              {R}adon type: maximal functions and vector-valued estimates. {\em J. Funct. Anal. } {\bf 2019}, {\em 8}, 2471--2521. 
 
      \bibitem[M24]{M24}
Mehmood, S.;   Mohammed, P.;   Kashuri, A.;  Chorfi, N.;   Mahmood, S.;  Yousif, M.  Some New Fractional Inequalities Defined Using cr-Log-h-Convex Functions and Applications.  {\em Symmetry } {\bf 2024}, {\em 16}, 407. 
   
       \bibitem[I04]{Iwaniec}
Iwaniec, H.;   Kowalski, E.   \textit{Analytic Number Theory}; American Mathematical Society: Providence, RI, USA, 2004. 

   
   
  \bibitem[SW58]{Stein-Weiss}
Stein, E.;  Weiss, G.    Fractional Integrals on $n$-Dimensional Euclidean Space. {\em J. Math. Mech.} {\bf 1958}, {\em 7}, 503--514.
    
 
\bibitem[S93]{Stein}
Stein, E.   \textit{Harmonic Analysis: Real-Variable Methods, Orthogonality and Oscillatory Integrals}; Princeton University Press: 
 NJ, USA, 1993.  
 
  \bibitem[HL28]{Hardy-Littlewood}
Hardy, G.;  Littlewood, J.    Some Properties of Fractional Integrals. {\em J. Math. Mech.} {\bf 1928}, {\em 27}, 565--606.
 
 
    
 
    \bibitem[S38]{Sobolev}
Sobolev, S. On a Theorem of Functional Analysis. {\em Mat. Sb.} {\bf 1938}, {\em 46}, 471--497. 

     \bibitem[W21]{Wang}
Wang, Z.      Stein-Weiss inequality on product spaces. {\em Rev. Mat. Iberoam.} {\bf 2021}, {\em 37}, 1641--1667. 



 
\end{thebibliography}
\end{document}